\documentclass{amsproc}
\usepackage{amsmath,amsthm,amsfonts,amssymb,amscd,amsbsy}
\usepackage{color}
\usepackage{comment}

\newtheorem{theorem}{Theorem}[section]
\newtheorem{lemma}[theorem]{Lemma}
\newtheorem{problem}[theorem]{Problem}

\theoremstyle{definition}

\newcommand{\nn}{\mathbb{N}}
\newcommand{\ee}{\varepsilon}

\numberwithin{equation}{section}

\begin{document}

\title[On injective tensor powers of $\ell_1$]{On injective tensor powers of $\ell_1$}

\author{R. M. Causey}
\address{Miami University, Department of Mathematics, Oxford, OH 45056, USA}
\email{causeyrm@miamioh.edu}

\author{E. M. Galego}
\address{University of S\~ao Paulo, Department of Mathematics, IME, Rua do Mat\~ao 1010,  S\~ao Paulo, Brazil}
\curraddr{Department of Mathematics and Statistics,}
\email{eloi@ime.usp.br}

\author{C. Samuel}
\address{Aix Marseille Universit\'e, CNRS, Centrale Marseille, I2M, Marseille, France}
\email{christian.samuel@univ-amu.fr}





\begin{abstract} In this paper we prove that the $3$-fold injective  tensor product $\ell_1 \widehat{\otimes}_\ee \ell_1 \widehat{\otimes}_\ee \ell_1 $     is not isomorphic to  any subspace of $\ell_1 \widehat{\otimes}_\ee \ell_1$.  This result provides a new solution to  a problem of Diestel  on the projective tensor products of $c_0.$ Moreover, this result implies that  for any infinite countable compact space $K,$  the $3$-fold projective tensor product $C(K) \widehat{\otimes}_\pi   C(K)\widehat{\otimes}_\pi C(K)$  is not isomorphic to any quotient of  $C(K) \widehat{\otimes}_\pi C(K)$.

\end{abstract}

\maketitle


\section{Introduction}
For standard Banach space terminology employed throughout the paper the reader is referred to \cite{JL} and \cite{Ry}. For $n\in\nn$, a tensor norm $\alpha$, and a Banach space $X$, let  $\widehat{\otimes}_\alpha^n X$ denote the $n$-fold $\alpha$-tensor product of $X$ with itself.

Very recently the authors solve a problem attributed to Diestel \cite[Theorem 1.3]{CGS} by proving that $\widehat{\otimes}_\pi^3 c_0$ is not isomorphic to $\widehat{\otimes}_\pi^2 c_0$. In the present paper we consider  two natural problems that arise from this result. The first  problem  is whether this result extends to  $C(K)$ spaces other than $c_0$, here  the space $C(K)$ will stand for the Banach space of all continuous, real-valued  functions on the compact Hausdorff space K and equipped with the  supremum norm. The first problem can be precisely stated as:
\begin{problem} \label{P1}Let $K$ be an infinite compact Hausdorff space. Is it true that $\widehat{\otimes}_\pi^3 C(K)$ is not isomorphic to $\widehat{\otimes}_\pi^2 C(K)$? 
\end{problem}
The second problem  is to know if the dual spaces of  $\widehat{\otimes}_\pi^3 c_0$      and  $\widehat{\otimes}_\pi^2 c_0$ are  isomorphic to each other. By using well-known properties of projective and injective tensor products \cite{Ry} this problem can be rewritten as follows:
\begin{problem} \label {P3} Is $\widehat{\otimes}_\ee^3 \ell_1$  isomorphic to    $\widehat{\otimes}_\ee^2 \ell_1$?  
\end{problem}
This last problem was proposed to us by Richard M. Aron to whom we are grateful for the interest shown in this research topic.

The  main goal of this paper is to present a negative  solution to Problem \ref{P3}. This follows directly the following  theorem.
\begin{theorem} \label{main} $\widehat{\otimes}_\ee^3 \ell_1$  is not isomorphic to any subspace of   $\widehat{\otimes}_\ee^2 \ell_1$.  

\end{theorem}
Observe that if $K$ is an infinite countable compact metric space, then it is well known that the dual space  of $C (K)$ is isomorphic to $\ell_1$ \cite[p.20]{JL}. Therefore if follows from Theorem \ref{main} that $\widehat{\otimes}_{\pi}^3 C(K)$   is not isomorphic to any quotient of  $\widehat{\otimes}_{\pi}^2 C(K)$. In particular,   Problem \ref{P1} has a positive solution when $K$ is an infinite countable compact metric space.

Theorem \ref{main}  also provides a new proof that $\widehat{\otimes}_\ee^2 \ell_1$  is not isomorphic to any subspace of   $\ell_1$ \cite[Corollary 2.1]{KP}. However we do not know how to solve:
\begin{problem} \label{fff} Suppose that for some $ m,n\in \nn$ with $m,n\geqslant 3$,  $\widehat{\otimes}_\ee^m \ell_1$ is isomorphic to $\widehat{\otimes}_\ee^n \ell_1$. Is it true that $m=n$?
\end{problem}

Of course it would be interesting to know if Problem \ref{P1} also has a positive solution   when  $K$ is the interval of  real numbers $[0, 1]$ or $K$ is  $\beta \mathbb N$,   the Stone-Cech compactification of the discrete set of natural numbers $\mathbb N$, see \cite{CFPV} to some geometric properties of the spaces  $C([0,1]) \widehat{\otimes}_\pi C([0,1])$ and $C(\beta \mathbb N) \widehat{\otimes}_\pi C(\beta \mathbb N)$. 

\

The fundamental property  used in \cite{CGS} concerned $\ell_2$ upper estimates on the branches of weakly null trees in the $2$-fold tensor product $\widehat{\otimes}_\pi^2 c_0$. Trees dualize nicely, but the dual property to upper $\ell_2$ estimates on weakly null trees in some Banach space $X$ is lower $\ell_2$ estimates on the branches of $\hbox{weak}^{*}$ null trees in $X^*$. Therefore the result  from \cite{CGS} does not yield that there is no isomorphic embedding of $\widehat{\otimes}_\ee^3 \ell_1$ into $\widehat{\otimes}_\ee^2 \ell_1$, because such an isomorphic embedding need not be $\hbox{weak}^{*}\hbox{-weak}^{*}$ continuous. Thus Theorem \ref{main} is not a trivial consequence of the result of \cite{CGS}. 

Also, the results of  \cite{CGS} were stated in terms of weakly  null trees, but the objects produced were weakly null arrays, which can be viewed as a special kind of weakly null tree. Since weakly null arrays are weakly null trees,  $\ell_2$ upper estimates on the branches of weakly null trees implies the same  estimates on the branches of weakly null arrays, but the converse need not hold \cite{AM}. Therefore the existence of weakly null arrays which do not satisfy a uniform $\ell_2$ upper  estimate is a stronger condition than the existence of weakly null trees. In the current work, we use the fact that \cite{CGS} produced arrays and not simply sequences, as this allows us to circumvent the difficulty that isomorphic embeddings need not be weak$^*$-weak$^*$ continuous.  The key step is noting that arrays are amendable to a certain differencing procedure, while the same differencing procedure cannot be applied to trees. This differencing is used here to overcome a difficulty not present in \cite{CGS}.

\section{Proof of Theorem \ref{main}}

For a Banach space $X$ and $n\in\nn$, a  family  $(x^k_i)_{i=1,k=1}^{\infty,n}$ of  $X$   is called an $n$-\emph{array}. For $C>0$, an $n$-array is said to be $C$-\emph{separated} provided that for any $1\leqslant k\leqslant n$ and any distinct $i,j\in \nn$, $\|x^k_i-x^k_j\|\geqslant C$.    

For a Banach space $X$ and $n\in\nn$, let $\delta_n(X)$ denote the infimum of $d>0$ such that for any $C>0$ and any bounded, $C$-separated  $n$-array $(x^k_i)_{i=1,k=1}^{\infty,n}$ in $ X,$  there exist $i_1<j_1<\ldots <i_n<j_n$ such that \[d\Bigl\|\sum_{k=1}^n (x^k_{i_k}-x^k_{j_k})\Bigr\|\geqslant Cn^{1/2}.\]  

Obviously if $X$ is isomorphic to a subspace of $Y$,  then $\sup_n \delta_n(X)/\delta_n(Y)<\infty$. More precisely, if $X,Z$ are isomorphic Banach spaces and $d_{BM}$ their Banach-Mazur distance, then $\delta_n(X)\leqslant d_{BM}\delta_n(Z)$ for all $n$, and if $Z$ is a closed subspace of $Y$, then $\delta_n(Z)\leqslant \delta_n(Y)$ for all $n\in\nn$.  Therefore we will prove Theorem \ref{main} by completing the next two lemmas. 

\begin{lemma} It holds that \[\sup_n\text{\ } \delta_n(\widehat{\otimes}_\ee^2 \ell_1)<\infty.\]

\label{l1}
\end{lemma}

\begin{lemma} It holds that  \[\inf_n \text{\ }\frac{\delta_n(\widehat{\otimes}_\ee^3 \ell_1)}{\log(n)}>0.\] 

\label{l2}
\end{lemma}

\begin{proof}[Proof of Lemma \ref{l1}] Let $(e_i)_i$ be the unit vector basis of $\ell_1, $ $(e_i^*)_i$ the biortho\-gonal sequence and, for every integer $k,$ $$F_k=\mathop{\rm span}\{ e_i\otimes e_j : \max\{i,j\}=k\}.$$
 It was proven in \cite{DK}  that the sequence of subspaces $(F_k)_k$   satisfies \color{black} a lower $\ell_2$ estimate. That is, there exist a constant $a>0$   such that \color{black} for any $0=q_0<q_1<\ldots$, any $n\in\nn$, and any $(y_i)_{i=1}^n\in \prod_{i=1}^n \text{span}\{F_j:q_{i-1}<j\leqslant q_i\}$,  \[a^2\Bigl\|\sum_{i=1}^n y_i\Bigr\|^2 \geqslant  \sum_{i=1}^n \|y_i\|^2.\]   

Fix $C>0$, $n\in\nn$, and a $C$-separated, bounded  $n$-array $(x^k_i)_{i=1,k=1}^{\infty,n}$ in $\widehat{\otimes}_\ee^2 \ell_1$.  By passing to subsequences $n$ times and relabeling, we may assume that for each $1\leqslant k\leqslant n$ and each  $(p,q) \in\mathbb{N}\times \mathbb{N}$, $\lim_i \langle e_p^*\otimes e^*_q, x^k_i\rangle$ exists. Then for $\ee>0$ and some appropriately chosen $i_1<j_1<\ldots <i_n<j_n$, $(x^k_{i_k}-x^k_{j_k})_{k=1}^n$ will be a small perturbation of a block sequence with respect to the blocking $(F_j)_{j=1}^\infty$ and will satisfy \[a\Bigl\|\sum_{k=1}^n (x^k_{i_k}-x^k_{j_k})\Bigr\| \geqslant \Bigl(\sum_{k=1}^n \|x^k_{i_k}-x^k_{j_k}\|^2\Bigr)^{1/2}-\ee \geqslant Cn^{1/2}-\ee.\] From this it follows that $\sup_n \delta_n(\widehat{\otimes}_\ee^2 \ell_1)\leqslant a$. 

\end{proof}

\begin{proof}[Proof of Lemma \ref{l2}]

For $1<n\in\nn$ and $1\leqslant k\leqslant n$, define \[t^n_k=\sum_{n+1-k\neq j=1}^n \frac{1}{n+1-j-k}e_j.\]  Define \[g^n_k=\frac{1}{2^n}\sum_{i=1}^{2^k}\sum_{j=(i-1)2^{n-k}+1}^{i2^{n-k}} (-1)^i e_j\in S_{\ell_1}.\]  Note that there exists a constant $0<\beta$ (independent of both $n$ and $k$) such that \[\beta\log(n) \leqslant \|t^n_k\|_{\ell_1}.\]  It was shown in \cite{CGS} that there exists a constant $\tau<\infty$ (independent of $n$) such that \[\Bigl\|\sum_{k=1}^n e_k\otimes t^n_k\otimes g^n_k\Bigr\|_{\widehat{\otimes}_\ee^3 \ell_1} \leqslant \tau n^{1/2}.\]  There the norm was computed with $T=\sum_{k=1}^n e_k\otimes t^n_k\otimes g^n_k$ treated as a member of $(\widehat{\otimes}_\pi^3 c_0)^*$, but this is equivalent to the norm in $\widehat{\otimes}_\ee^3 \ell_1$.

Define the array $(x^k_i)_{i=1,k=1}^{\infty,n}$ by letting $x^k_i=e_i\otimes t^n_k\otimes g^n_k$. By $1$-unconditionality of the $\ell_1$ basis, for any $1\leqslant k\leqslant n$ and any distinct $i,j\in\nn$, \[\|x^k_i-x^k_j\|\geqslant \|x^k_i\|=\|e_i\|\|t^n_k\|\|g^n_k\|\geqslant \beta \log(n).\] Therefore  the array $(x^k_i)_{i=1,k=1}^{\infty, n}$ is $C=\beta \log(n)$-separated.

By $1$-subsymmetry of the $\ell_1$ basis, it follows that for any $i_1<i_2<\ldots <i_n<j_n$, \begin{align*} \Bigl\|\sum_{k=1}^n x^k_{i_k}-x^k_{j_k}\Bigr\| & \leqslant \Bigl\|\sum_{k=1}^n x^k_{i_k}\Bigr\|+\Bigl\|\sum_{k=1}^n x^k_{j_k}\Bigr\| \\ & = 2\Bigl\|\sum_{k=1}^n x^k_k\Bigr\| =2\Bigl\|\sum_{k=1}^n e_k\otimes t^n_k\otimes g^n_k\Bigr\|\leqslant 2\tau n^{1/2}.\end{align*}

Therefore $d 2\tau \geqslant \beta \log(n)n^{1/2}.$

From this it follows that $\delta_n (\widehat{\otimes}_\ee^3 \ell_1) \geqslant \frac{\beta \log(n)}{2\tau}$.  Since neither $\beta$ nor $\tau $ depends on $n$, we are done.

\end{proof}

\end{document}